\documentclass[10pt]{article}

\usepackage{ucs}
\usepackage{amssymb}
\usepackage{amsthm}
\usepackage{amsmath}
\usepackage{latexsym}
\usepackage[cp1251]{inputenc}
\usepackage[english]{babel}
\usepackage{graphicx}
\usepackage{wrapfig}
\usepackage{caption}
\usepackage{subcaption}
\usepackage{txfonts}
\usepackage{lmodern}
\usepackage{mathrsfs}
\usepackage{eufrak}
\usepackage{algorithmicx}
\usepackage{algpseudocode}
\usepackage{algorithm}
\usepackage{dsfont}
\usepackage{enumerate}
\usepackage{hyphenat}
\usepackage[hidelinks]{hyperref}
\usepackage{mathtools}

\usepackage[left=3cm,right=3cm,top=2cm,bottom=2.5cm,bindingoffset=0cm]{geometry}

\title{New Bounds for the Acyclic Chromatic Index}
\date{}
\author{Anton Bernshteyn\thanks{Supported by the Illinois Distinguished Fellowship.}\\ University of Illinois at Urbana-Champaign}

\newtheorem{theo}{Theorem}

\newtheorem{lemma}[theo]{Lemma}

\newtheorem{corl}[theo]{Corollary}
\newtheorem{conj}[theo]{Conjecture}

\theoremstyle{definition}
\newtheorem{defn}[theo]{Definition}

\theoremstyle{remark}
\newtheorem{remk}[theo]{Remark}

\begin{document}
		\maketitle
		
		\begin{abstract}
			An edge coloring of a graph $G$ is called an \emph{acyclic edge coloring} if it is proper and every cycle in $G$ contains edges of at least three different colors. The least number of colors needed for an acyclic edge coloring of $G$ is called the \emph{acyclic chromatic index} of $G$ and is denoted by $a'(G)$. Fiam\v{c}ik \cite{Fiamcik} and independently Alon, Sudakov, and Zaks \cite{Alon2} conjectured that $a'(G) \leq \Delta(G)+2$, where $\Delta(G)$ denotes the maximum degree of $G$. The best known general bound is $a'(G)\leq 4(\Delta(G)-1)$ due to Esperet and Parreau \cite{Esperet}. We apply a generalization of the Lov\'{a}sz Local Lemma to show that if $G$ contains no copy of a given bipartite graph $H$, then $a'(G) \leq 3\Delta(G)+o(\Delta(G))$. Moreover, for every $\varepsilon>0$, there exists a constant $c$ such that if $g(G)\geq c$, then $a'(G)\leq(2+\varepsilon)\Delta(G)+o(\Delta(G))$, where $g(G)$ denotes the girth of $G$.
		\end{abstract}
		
		\section{Introduction}
		
		All graphs considered here, unless indicated otherwise, are finite, undirected, and simple. An edge coloring of a graph $G$ is called an \emph{acyclic edge coloring} if it is proper (i.e., adjacent edges receive different colors) and every cycle in $G$ contains edges of at least three different colors (there are no \emph{bichromatic cycles} in $G$). The least number of colors needed for an acyclic edge coloring of $G$ is called the \emph{acyclic chromatic index} of $G$ and is denoted by $a'(G)$. The notion of acyclic (vertex) coloring was first introduced by Gr\"{u}nbaum \cite{Grunbaum}. The edge version was first considered by Fiam\v{c}ik \cite{Fiamcik}, and independently by Alon, McDiarmid, and Reed \cite{Alon1}.
		
		As for many other graph parameters, it is quite natural to ask for an upper bound on the acyclic chromatic index of a graph $G$ in terms of its maximum degree $\Delta(G)$. Since $a'(G)\geq \chi'(G) \geq \Delta(G)$, where $\chi'(G)$ denotes the ordinary chromatic index of $G$, this bound must be at least linear in $\Delta(G)$. The first linear bound was given by Alon et al. \cite{Alon1}, who showed that $a'(G)\leq 64 \Delta(G)$. Although it resolved the problem of determining the order of growth of $a'(G)$ in terms of $\Delta(G)$, it was conjectured that the sharp bound should be lower.
		
		\begin{conj}[Fiam\v{c}ik \cite{Fiamcik}; Alon, Sudakov, Zaks \cite{Alon2}]\label{conj:AECC}
			For every graph $G$, $a'(G) \leq \Delta(G)+2$.
		\end{conj}
		
		Note that the bound in Conjecture \ref{conj:AECC} is only one more than Vizing's bound on the chromatic index of $G$. However, this elegant conjecture is still far from being proven.
		
		The first major improvement to the bound $a'(G)\leq 64 \Delta(G)$ was made by Molloy and Reed \cite{Molloy}, who proved that $a'(G) \leq 16 \Delta(G)$. This bound remained the best for a while, until Ndreca, Procacci, and Scoppola \cite{Ndreca} managed to improve it to $a'(G)\leq \left\lceil 9.62(\Delta(G)-1)\right\rceil$. This estimate was recently lowered further to $a'(G)\leq 4(\Delta(G)-1)$ by Esperet and Parreau \cite{Esperet}.
		
		All the bounds mentioned above were derived using probabilistic arguments, and recent progress was stimulated by discovering more sophisticated and powerful analogues of the Lov\'{a}sz Local Lemma, namely the stronger version of the LLL due to Bissacot, Fern\'{a}ndez, Procacci, and Scoppola \cite{Bissacot} and the entropy compression method of Moser and Tardos \cite{Moser}.
		
		The probability that a cycle would become bichromatic in a random coloring is less if the cycle is longer. Thus, it should be easier to establish better bounds on the acyclic chromatic index for graphs with high enough girth. Indeed, Alon et al. \cite{Alon2} showed that if $g(G) \geq c_1 \Delta(G) \log\Delta(G)$, where $c_1$ is some universal constant, then $a'(G) \leq \Delta(G)+2$. They also proved that if $g(G)\geq c_2 \log\Delta(G)$, then $a'(G) \leq 2\Delta(G)+2$. This was lately improved by Muthu, Narayanan, and Subramanian \cite{Muthu} in the following way: For every $\varepsilon > 0$, there exists a constant $c$ such that if $g(G)\geq c \log \Delta(G)$, then $a'(G)\leq (1+\varepsilon)\Delta(G)+o(\Delta(G))$.
		
		We are now turning to the case when $g(G)$ is bounded below by some constant independent of $\Delta(G)$, which will be the main topic of this paper. The first bounds of such type were given by Muthu et al. \cite{Muthu}, who proved that $a'(G) \leq 9\Delta(G)$ if $g(G)\geq 9$, and $a'(G) \leq 4.52\Delta(G)$ if $g(G)\geq 220$. Esperet and Parreau \cite{Esperet} not only improved both these estimates even in the case of arbitrary $g(G)$, but they also showed that $a'(G) \leq \left\lceil 3.74(\Delta(G)-1)\right\rceil$ if $g(G)\geq 7$, $a'(G) \leq \left\lceil 3.14(\Delta(G)-1)\right\rceil$ if $g(G)\geq 53$, and, in fact, for every $\varepsilon > 0$, there exists a constant $c$ such that if $g(G)\geq c$, then $a'(G) \leq (3+\varepsilon)\Delta(G) + o(\Delta(G))$.
		
		The result that we present here consists in further improvement of the latter bounds. Namely, we establish the following.
		
		\begin{theo}\label{theo:no4}
			Let $G$ be a graph with maximum degree $\Delta$ and let $H$ be some bipartite graph. If $G$ does not contain $H$ as a subgraph, then $a'(G)\leq 3\Delta + o(\Delta)$.
		\end{theo}
		\begin{remk}
			In our original version of Theorem \ref{theo:no4} we considered only the case where $H$ was the $4$-cycle. That almost the same proof in fact works for any bipartite $H$ was observed by Esperet and de Verclos.
		\end{remk}
		\begin{remk}
			The function $o(\Delta)$ in the statement of Theorem \ref{theo:no4} depends on $H$. In fact, our proof shows that for the complete bipartite graph $K_{k,k}$, it is of the order $O\left(\Delta^{1-1/2k}\right)$.
		\end{remk}
		\begin{theo}\label{theo:bigg}
			For every $\varepsilon > 0$, there exists a constant $c$ such that for every graph $G$ with maximum degree $\Delta$ and $g(G) \geq c$, we have $a'(G) \leq (2+\varepsilon)\Delta + o(\Delta)$.
		\end{theo}
		\begin{remk}
			The bound of the last theorem was recently improved further to $a'(G) \leq (1+\varepsilon)\Delta + o(\Delta)$ by Cai, Perarnau, Reed, and Watts \cite{Cai} using a different (and much more sophisticated) argument.
		\end{remk}
		
		To prove Theorems \ref{theo:no4} and \ref{theo:bigg}, we use a generalization of the Lov\'{a}sz Local Lemma that we call the \emph{Local Cut Lemma} (the LCL for short). The LCL is inspired by recent combinatorial applications of the entropy compression method, although its proof is probabilistic and does not use entropy compression. Several examples of applying the LCL and its proof can be found in \cite{Bernshteyn}. We provide all the required definitions and the statement of the LCL in Section \ref{sec:LCL}. In Section \ref{sec:no4} we prove Theorem \ref{theo:no4}, and in Section \ref{sec:bigg} we prove Theorem \ref{theo:bigg}.
		
		\section{The Local Cut Lemma}\label{sec:LCL}
		
		Roughly speaking, the Local Cut Lemma asserts that if a random set of vertices in a directed graph is cut out by a set of edges which is ``locally small'', then this set of vertices has to be ``large'' with positive probability. To state it rigorously, we will need some definitions.
		
		By a \emph{digraph} we mean a finite directed multigraph. Suppose that $D$ is a digraph with vertex set $V$ and edge set $E$. For $x$, $y \in V$, let $E(x,y) \subseteq E$ denote the set of all edges with tail $x$ and head $y$.
		
		A digraph $D$ is \emph{simple} if for all $x$, $y \in V$, $|E(x,y)| \leq 1$. If $D$ is simple and $|E(x,y)| = 1$, then the unique edge with tail $x$ and head $y$ is denoted by $xy$ (or sometimes $(x,y)$). For an arbitrary digraph $D$, let $D^s$ denote its \emph{underlying simple digraph}, i.e., the simple digraph with vertex set $V$ in which $xy$ is an edge if and only if $E(x,y) \neq \emptyset$. Denote the edge set of $D^s$ by $E^s$. For a set $F \subseteq E$, let $F^s \subseteq E^s$ be the set of all edges $xy \in E^s$ such that $F \cap E(x,y) \neq \emptyset$.
		
		A set $A \subseteq V$ is \emph{out-closed} (resp. \emph{in-closed}) if for all $xy \in E^s$, $x \in A$ implies $y \in A$ (resp. $y \in A$ implies $x \in A$).
		
		\begin{defn}
			Let $D$ be a digraph with vertex set $V$ and edge set $E$ and let $A \subseteq V$ be an out-closed set of vertices. A set $F \subseteq E$ of edges is an \emph{$A$-cut} if $A$ is in-closed in $D^s - F^s$.
		\end{defn}
		
		In other words, an $A$-cut $F$ has to contain at least one edge $e \in E(x,y)$ for each pair $x$, $y$ such that $x \in V \setminus A$, $y \in A$, and $xy \in E^s$.
		
		To understand the motivation behind the LCL, suppose that we are given a random out-closed set of vertices $A \subseteq V$ and a random $A$-cut $F \subseteq E$ in a digraph $D$. Since $A$ is out-closed, for every $xy \in E^s$,
		$$
			\Pr(x \in A) \leq \Pr(y \in A).
		$$
		We would like to establish a similar inequality in the other direction. More precisely, we want to find a function $\omega : E^s \to \mathbb{R}_{\geq 1}$ such that for all $xy \in E^s$,
		\begin{equation}\label{eq:upperbound}
			\Pr(y\in A) \leq \Pr(x \in A) \cdot \omega(xy).
		\end{equation}
		Note that if we can prove (\ref{eq:upperbound}) for some function $\omega$, then for all $xy \in E^s$,
		$$
			\Pr(x \in A) \geq \frac{\Pr(y \in A)}{\omega(xy)},
		$$
		so $\Pr(y \in A) > 0$ implies $\Pr(x \in A) > 0$. More generally, we say that a vertex $z \in V$ is \emph{reachable} from $x \in V$ if $D$ (or, equivalently, $D^s$) contains a directed $xz$-path. We can give the following definition.
		
		\begin{defn}
			Let $D$ be a digraph with vertex set $V$ and edge set $E$. Suppose that $\omega : E^s \to \mathbb{R}_{\geq 1}$ is an assignment of nonnegative real numbers to the edges of $D^s$. For $x$, $z \in V$ such that $z$ is reachable from $x$, define
			$$
				\underline{\omega}(x,z) \coloneqq \min \left\{\prod_{i=1}^k \omega(z_{i-1}z_i) \,:\, \text{$x = z_0\longrightarrow z_1\longrightarrow$\ldots$\longrightarrow z_k = z$ is a directed $xz$-path in $D^s$}\right\}.
			$$
		\end{defn}
		
		Then (\ref{eq:upperbound}) yields that for all $x$, $z \in V$ such that $z$ is reachable from $x$,
		$$
			\Pr(x \in A) \geq \frac{\Pr(z \in A)}{\underline{\omega}(x,z)}.
		$$
		In particular, $\Pr(z \in A) > 0$ implies $\Pr(x \in A) > 0$.
		
		How can we show that (\ref{eq:upperbound}) holds for some function $\omega$? Since $F$ is an $A$-cut, it would be enough to prove that for each $e \in E(x,y)$, the following probability is small:
		\begin{equation}\label{eq:prob}
			\Pr(e \in F \vert y \in A).
		\end{equation}
		In fact, it suffices to have
		$$
			\sum_{e \in E(x,y)} \Pr(e \in F \vert y \in A) < 1.
		$$
		Unfortunately, probability (\ref{eq:prob}) is usually hard to estimate directly. However, it might turn out that for some other vertex $z$ reachable from $y$, we can give a good upper bound on the following similar probability:
		\begin{equation}\label{eq:prob1}
			\Pr(e \in F \vert z \in A).
		\end{equation}
		Of course, the farther $z$ is from $y$ in the digraph, the less useful an upper bound on (\ref{eq:prob1}) would be. The following definition captures this trade-off.
		
		\begin{defn}\label{defn:risk}
			Let $D$ be a digraph with vertex set $V$ and edge set $E$. Suppose that $A \subseteq V$ is a random out-closed set of vertices and let $F \subseteq E$ be a random $A$-cut. Fix a function $\omega : E^s \to \mathbb{R}_{\geq 1}$. For $xy \in E^s$, $e \in E(x,y)$, and a vertex $z \in V$ reachable from $y$, let
			\begin{equation}\label{eq:risk}
				\rho_\omega(e, z) \coloneqq \Pr(e \in F \vert z \in A) \cdot \underline{\omega}(x, z).
			\end{equation}
			For $e \in E(x,y)$, define the \emph{risk} to $e$ as
			$$
				\rho_\omega(e) \coloneqq \min_z \rho_\omega(e, z),
			$$
			where $z$ ranges over all vertices reachable from $y$.
		\end{defn}	
		
		\begin{remk}
			For random events $B$, $B'$, the conditional probability $\Pr(B'\vert B)$ is only defined if $\Pr(B) > 0$. For convenience, we adopt the following notational convention in Definition~\ref{defn:risk}: If $B$ is a random event and $\Pr(B) = 0$, then $\Pr(B'\vert B) = 0$ for all events $B'$. Note that this way the crucial equation $\Pr(B' \vert B) \cdot \Pr(B) = \Pr(B' \cap B)$ is satisfied even when $\Pr(B) = 0$, and this is the only property of conditional probability we will use.
		\end{remk}
		
		We are now ready to state the LCL.
		
		\begin{lemma}[Local Cut Lemma \cite{Bernshteyn}]\label{lemma:LCL}
			Let $D$ be a digraph with vertex set $V$ and edge set $E$. Suppose that $A \subseteq V$ is a random out-closed set of vertices and let $F \subseteq E$ be a random $A$-cut. If a function $\omega : E^s \to \mathbb{R}_{\geq 1}$ satisfies the following inequality for all $xy \in E^s$:
			\begin{equation}\label{eq:LCL}
				\omega(xy) \geq 1 + \sum_{e \in E(x,y)} \rho_\omega(e),
			\end{equation}
			then for all $xy \in E^s$,
			$$
				\Pr(y \in A) \leq \Pr(x \in A) \cdot\omega(xy).
			$$
		\end{lemma}
		
		\begin{corl}\label{corl:positive}
			Let $D$, $A$, $F$, $\omega$ be as in Lemma~\ref{lemma:LCL}. Let $x$, $z \in V$ be such that $z$ is reachable from $x$ and suppose that $\Pr(z \in A) > 0$. Then
			$$
			\Pr(x \in A) \geq \frac{\Pr(z \in A)}{\underline{\omega}(x,z)} > 0.
			$$
		\end{corl}
		
		One particular class of digraphs often appearing in applications of the LCL is the class of the \emph{hypercube digraphs}. For a finite set $I$, the vertices of the hypercube digraph $Q(I)$ are the subsets of $I$ and its edges are of the form $(S \cup \{i\}, S)$ for each $S \subset I$ and $i \in I \setminus S$ (in particular, $Q(I)$ is simple). Note that if $S$, $T \subseteq I$, then $T$ is reachable from $S$ in $Q(I)$ if and only if $T \subseteq S$. Moreover, if $T \subseteq S$, then any directed $(S,T)$-path has length exactly $|S \setminus T|$. Therefore, if we assume that $\omega(S \cup \{i\}, S) = \omega \in \mathbb{R}_{\geq 1}$ is a fixed constant, then
		$$
		\underline{\omega}(S,T) = \omega^{|S \setminus T|}
		$$		
		for all $T \subseteq S \subseteq I$.
		
		\section{Graphs with a forbidden bipartite subgraph}\label{sec:no4}
		
		\subsection{Combinatorial lemmata}
		
		For this section we assume that a bipartite graph $H$ is fixed. In particular, all constants that we mention depend on $H$. We will use the following version of the K\H{o}vari--S\'{o}s--Tur\'{a}n theorem.
		
		\begin{theo}[K\H{o}vari, S\'{o}s, Tur\'{a}n \cite{Kovari}]\label{theo:Kovari}
			Let $G$ be a graph with $n$ vertices and $m$ edges that does not contain the complete bipartite graph $K_{k, k}$ as a subgraph. Then $m \leq O(n^{2-1/k})$ (assuming that $n \to \infty$).
		\end{theo}
		\begin{corl}\label{corl:noH}
			There exist positive constants $\alpha$ and $\delta$ such that if a graph $G$ with $n$ vertices and $m$ edges does not contain $H$ as a subgraph, then $m \leq \alpha n^{2- \delta}$.
		\end{corl}
		
		In what follows we fix the constants $\alpha$ and $\delta$ from the statement of Corollary \ref{corl:noH}. We say that the \emph{length} of a path $P$ is the number of edges in it. Using Corollary \ref{corl:noH}, we obtain the following.
		
		\begin{lemma}\label{lemma:3path}
			There is a positive constant $\beta$ such that the following holds. Let $G$ be a graph with maximum degree $\Delta$ that does not contain $H$ as a subgraph. Then for any two vertices $u$, $v \in V(G)$, the number of $uv$-paths of length $3$ in $G$ is at most $\beta\Delta^{2-\delta}$.
		\end{lemma}
		\begin{proof}
			Suppose that $u$---$x$---$y$---$v$ is a $uv$-path of length $3$ in $G$. Then $x \in N_G(u)$, $y \in N_G(v)$, and hence $xy \in E(G[N_G(u)\cup N_G(v)])$. Note that any edge $xy \in E(G[N_G(u)\cup N_G(v)])$ can possibly give rise to at most two different $uv$-paths of length $3$ (namely $u$---$x$---$y$---$v$ and $u$---$y$---$x$---$v$). Therefore, the number of $uv$-paths of length $3$ in $G$ is not greater than $2|E(G[N_G(u)\cup N_G(v)])|$. Since $|V(G[N_G(u)\cup N_G(v)])| \leq 2\Delta$, by Corollary \ref{corl:noH} we have that $|E(G[N_G(u)\cup N_G(v)])| \leq \alpha (2\Delta)^{2-\delta}$, so the number of $uv$-paths of length $3$ in $G$ is at most $\left(2^{3-\delta}\alpha\right)\Delta^{2-\delta}$.
		\end{proof}
		
		In what follows we fix the constant $\beta$ from the statement of Lemma \ref{lemma:3path}. The following fact is crucial for our proof.
		
		\begin{lemma}\label{lemma:cycles}
			Let $G$ be a graph with maximum degree $\Delta$ that does not contain $H$ as a subgraph. Then for any edge $e \in E(G)$ and for any integer $k \geq 4$, the number of cycles of length $k$ in $G$ that contain $e$ is at most $\beta\Delta^{k-2-\delta}$.
		\end{lemma}
		\begin{proof}
			Suppose that $e = uv \in E(G)$. Note that the number of cycles of length $k$ that contain $e$ is not greater than the number of $uv$-paths of length $k-1$. Consider any $uv$-path $u$---$x_1$---\ldots---$x_{k-2}$---$v$ of length $k-1$. Then $u$---$x_1$---\ldots---$x_{k-4}$ is a path of length $k-4$, and $x_{k-4}$---$x_{k-3}$---$x_{k-2}$---$v$ is a path of length $3$. There are at most $\Delta^{k-4}$ paths of length $k-4$ starting at $u$, and, given a path $u$---$x_1$---\ldots---$x_{k-4}$, the number of $x_{k-4}v$-paths of length $3$ is at most $\beta \Delta^{2-\delta}$. Hence the number of $uv$-paths of length $k-1$ is at most $\Delta^{k-4} \cdot \beta \Delta^{2-\delta} = \beta \Delta^{k-2-\delta}$.
		\end{proof}
		
		\subsection{Probabilistic set-up}\label{subsec:prob}
		
		Let $G$ be a graph with maximum degree $\Delta$ that does not contain $H$ as a subgraph. Fix some constant $c$ and let $\varphi$ be a $(2+c)\Delta$-edge coloring of $G$ chosen uniformly at random. Consider the hypercube digraph $Q \coloneqq Q(E(G))$. Define a set $A \subseteq V(Q)$ as follows:
		$$
			S \in A \,\vcentcolon\Longleftrightarrow\, \text{$\varphi$ is an acyclic edge coloring of $G[S]$},
		$$ 
		where $G[S]$ denotes the subgraph of $G$ obtained by deleting all the edges outside $S$. Note that $A$ is out-closed in $Q$. Moreover, $\emptyset \in A$ with probability $1$. On the other hand, $E(G) \in A$ if and only if $\varphi$ is an acyclic edge coloring of $G$. Therefore, if we can apply the LCL here, then Corollary~\ref{corl:positive} will imply that $\varphi$ is an acyclic edge coloring of $G$ with positive probability, and hence $a'(G) \leq (2 + c)\Delta$.
		
		Call a cycle $C$ of length $2t$ \emph{$\varphi$-bichromatic} if $C= e_1$, $e_2$, \ldots, $e_{2t}$ and $\varphi(e_{2i-1}) = \varphi(e_{2t-1})$, $\varphi(e_{2i}) = \varphi(e_{2t})$ for all $1 \leq i \leq t-1$. Consider an arbitrary edge $(S \cup \{e\}, S)$ of $Q$. If $S \in A$, but $S \cup \{e\} \not \in A$, then either there exists an edge $e' \in S$ adjacent to $e$ such that $\varphi(e) = \varphi(e')$, or there exists a $\varphi$-bichromatic cycle $C \subseteq S \cup \{e\}$ such that $e \in C$. This observation motivates the following construction. Let $\widetilde{Q}$ be the digraph with $\widetilde{Q}^s = Q$ such that for each $(S \cup \{e\}, S) \in E(Q)$, the edges of $\widetilde{Q}$ between $S \cup \{e\}$ and $S$ are of the following two kinds:
		\begin{enumerate}
			\item $f_C$ for each cycle $C$ of even length passing through $e$;
			\item one additional edge $f$.
		\end{enumerate}
		Let
		$$
			f \in F\,\vcentcolon\Longleftrightarrow\, \text{there exists an edge $e'$ adjacent to $e$ such that $\varphi(e) = \varphi(e')$},
		$$
		and
		$$
			f_C \in F \,\vcentcolon\Longleftrightarrow\, \text{$C$ is $\varphi$-bichromatic}.
		$$
		The above observation implies that $F$ is an $A$-cut in $\widetilde{Q}$.
		
		To apply the LCL, it remains to estimate the risk to each edge of $\widetilde{Q}$. Denote the edge set of $\widetilde{Q}$ by $E$ and consider any $(S \cup \{e\}, S) \in E^s$. Note that a vertex of $\widetilde{Q}$ reachable from $S$ is just a subset of $S$. First consider the edge $f\in (S \cup \{e\}, S)$ of the second kind. For a given edge $e'$ adjacent to $e$, we have
		$$
			\Pr(\varphi(e) = \varphi(e') \vert S \in A) \leq \frac{1}{(2+c)\Delta},
		$$
		since $e \not \in S$ and the colors of different edges are independent. (Note that we might have a strict inequality if $\Pr(S \in A) = 0$, since then, according to our convention, $\Pr(\varphi(e) = \varphi(e') \vert S \in A) = 0$ as well.) Thus, we have
		\begin{align*}
			\Pr(f \in F \vert S \in A) &\leq\sum_{e' \in N_G(e)}  \Pr(\varphi(e) = \varphi(e') \vert S \in A) \\
			&\leq 2\Delta \cdot \frac{1}{(2+c)\Delta} = \frac{2}{2+c},
		\end{align*}
		where $N_G(e)$ denotes the set of edges in $G$ adjacent to $e$. Note that $|(S \cup \{e\}) \setminus S| = 1$, so, if we assume that $\omega \in \mathbb{R}_{\geq 1}$ is a fixed constant,
		$$
			\underline{\omega}(S \cup \{e\}, S) = \omega.
		$$
		Therefore,
		$$
			\rho_\omega(f) \leq \rho_\omega(f, S) \leq \frac{2\omega}{2+c}.
		$$
		
		Now consider any edge $f_C \in E(S \cup \{e\}, S)$ corresponding to a cycle $C \ni e$ of length $2t$.		Suppose that $C = e_1$, $e_2$, \ldots, $e_{2t}$, where $e_1 = e$. Then $f_C \in F$ if and only if $\varphi(e_{2i-1}) = \varphi(e_{2t-1})$ and $\varphi(e_{2i}) = \varphi(e_{2t})$ for all $1 \leq i \leq t-1$. Even if the colors of $e_{2t-1}$ and $e_{2t}$ are fixed, the probability of this happening is $1/((2+c)\Delta)^{2t-2}$. Keeping this observation in mind, let $C' \coloneqq \{e_1, e_2, \ldots, e_{2t-2}\}$. Then
		$$
			\Pr(f_C \in F \vert S \setminus C' \in A) \leq \frac{1}{((2+c)\Delta)^{2t-2}}.
		$$
		Also, $|(S \cup \{e\})\setminus (S \setminus C')| \leq |C'| = 2t-2$, so if we assume that $\omega \in \mathbb{R}_{\geq 1}$ is a constant, then
		$$
			\underline{\omega}(S \cup \{e\}, S \setminus C') \leq \omega^{2t-2}.
		$$
		Therefore,
		$$
			\rho_\omega(f_C) \leq \rho_\omega(f_C, S \setminus C') \leq \frac{\omega^{2t-2}}{((2+c)\Delta)^{2t-2}}.
		$$
		
		We are ready to apply the LCL. Assuming that $\omega \in \mathbb{R}_{\geq 1}$ is a constant and using Lemma~\ref{lemma:cycles}, it is enough to show
		\begin{equation}\label{eq:acyc}
			\omega \geq 1 + \sum_{t = 2}^\infty \frac{\beta\Delta^{2t-2-\delta} \omega^{2t-2}}{((2+c)\Delta)^{2t-2}} + \frac{2\omega}{2+c} = 1 + \beta\Delta^{-\delta}\frac{\left(\omega/(2+c)\right)^2}{1-\left(\omega/(2+c)\right)^2}+\frac{2\omega}{2+c},
		\end{equation}
		where the last equality holds under the assumption that $\omega/(2+c) < 1$. If we denote $y = \omega/(2+c)$, then (\ref{eq:acyc}) turns into
		\begin{align}\label{eq:acyc1}
			c \geq \frac{1}{y} + \beta \Delta^{-\delta} \frac{y^2}{1-y^2}.
		\end{align}
		Now if $c = 1 + \varepsilon$ for any given $\varepsilon > 0$, then we can take $1/y = 1 + \varepsilon/2$. For this particular value of $y$, we have
		$$
			\beta \Delta^{-\delta} \frac{y^2}{1-y^2} \xrightarrow[\Delta \to \infty]{} 0,
		$$
		so for $\Delta$ large enough, $\beta \Delta^{-\delta} y^2/(1-y^2) \leq \varepsilon/2$, and (\ref{eq:acyc1}) is satisfied. This observation completes the proof of Theorem \ref{theo:no4}. A more precise calculation shows that (\ref{eq:acyc1}) can be satisfied for $c = 1 + \varepsilon$ as long as $\varepsilon \geq \gamma\Delta^{-\delta/2}$ for some absolute constant $\gamma$.
		
		\section{Graphs with large girth}\label{sec:bigg}
		
		\subsection{Breaking short cycles}
		
		The proof of Theorem \ref{theo:bigg} proceeds in two steps. Assuming that the girth of $G$ is large enough, we first show that there is a proper edge coloring of $G$ by $(2+\varepsilon/2)\Delta$ colors with no ``short'' bichromatic cycles (where ``short'' means of length roughly $\log \Delta$). Then we use the remaining $\varepsilon \Delta/2$ colors to break all the ``long'' bichromatic cycles.
		
		We start with the following observations analogous to Lemmata \ref{lemma:3path} and \ref{lemma:cycles}.
		
		\begin{lemma}\label{lemma:rpath}
			Let $G$ be a graph with maximum degree $\Delta$ and girth $g > 2r$, where $r \geq 2$. Then for any two vertices $u$, $v \in V(G)$, the number of $uv$-paths of length $r$ in $G$ is at most $1$.
		\end{lemma}
		\begin{proof}
			If there are two $uv$-paths of length $r$, then their union forms a closed walk of length $2r$, which means that $G$ contains a cycle of length at most $2r$.
		\end{proof}
		
		\begin{lemma}\label{lemma:cycles1}
			Let $G$ be a graph with maximum degree $\Delta$ and girth $g > 2r$, where $r \geq 2$. Then for any edge $e \in E(G)$ and for any integer $k \geq 4$, the number of cycles of length $k$ in $G$ that contain $e$ is at most $\Delta^{k-r-1}$.
		\end{lemma}
		\begin{proof}
			Suppose that $e = uv \in E(G)$. Note that the number of cycles of length $k$ that contain $e$ is not greater than the number of $uv$-paths of length $k-1$. Consider any $uv$-path $u$---$x_1$---\ldots---$x_{k-2}$---$v$ of length $k-1$. Then $u$---$x_1$---\ldots---$x_{k-r-1}$ is a path of length $k-r-1$, and $x_{k-r-1}$---$x_{k-r}$---\ldots---$x_{k-2}$---$v$ is a path of length $r$. There are at most $\Delta^{k-r-1}$ paths of length $k-r-1$ starting at $u$, and, given a path $u$---$x_1$---\ldots---$x_{k-r-1}$, the number of $x_{k-r-1}v$-paths of length $r$ is at most $1$. Hence the number of $uv$-paths of length $k-1$ is at most $\Delta^{k-r-1}$.		
		\end{proof}
		
		\begin{lemma}\label{lemma:short}
			For every $\varepsilon > 0$, there exists a positive constant $a_\varepsilon$ such that the following holds. Let $G$ be a graph with maximum degree $\Delta$ and girth $g > 2r$, where $r \geq 2$. Then there is a proper edge coloring of $G$ using at most $(2+\varepsilon)\Delta + o(\Delta)$ colors that contains no bichromatic cycles of length at most $2L$, where $L \coloneqq a_\varepsilon (r-2)\log\Delta + 1$.
		\end{lemma}
		\begin{proof}
			We work in a probabilistic setting similar to the one used in the proof of Theorem \ref{theo:no4} (see Subsection \ref{subsec:prob} for the notation used), but this time
			$$
				S \in A \,\vcentcolon\Longleftrightarrow\, \text{$\varphi$ is a proper edge coloring of $G[S]$ with no bichromatic cycles of length at most $2L$}.
			$$			
			Then, taking into account Lemma \ref{lemma:cycles1}, (\ref{eq:acyc}) turns into
			\begin{align}\label{eq:short}
				\omega \geq 1 + \sum_{t = r+1}^L \frac{\Delta^{2t-r-1} \omega^{2t-2}}{((2+c)\Delta)^{2t-2}} + \frac{2\omega}{2+c} = 1 + \Delta^{-r+1} \sum_{t = r+1}^L \left(\frac{\omega}{2+c}\right)^{2t-2} + \frac{2\omega}{2+c}.
			\end{align}
			If $y \coloneqq \omega/(2+c)$, then (\ref{eq:short}) becomes
			\begin{align*}
				c \geq \frac{1}{y} + \Delta^{-r+1} \sum_{t=r+1}^L y^{2t-3}.
			\end{align*}
			Note that if $y > 1$ and $L\leq \Delta$, we have
			\begin{align*}
				\sum_{t=r+1}^L y^{2t-3} \leq \sum_{t=r+1}^L y^{2L-3} = (L-r) y^{2L-3} \leq \Delta y^{2L-3},
			\end{align*}
			so it is enough to get
			\begin{align*}
				c \geq \frac{1}{y} + \Delta^{-r+2} y^{2L-3}.
			\end{align*}
			Now take $y = 2/\varepsilon$ and $c = \varepsilon$. We need
			\begin{align*}
				\frac{\varepsilon}{2} \geq \Delta^{-r+2} \left(\frac{2}{\varepsilon}\right)^{2L-3},
			\end{align*}
			i.e.
			\begin{align*}
				L \leq \left(2\log\frac{2}{\varepsilon}\right)^{-1} (r-2)\log\Delta + 1,
			\end{align*}
			and we are done.
		\end{proof}
		
		\subsection{Breaking long cycles}
		
		To deal with ``long'' cycles we need a different random procedure. A similar procedure was analysed in \cite{Muthu} using the LLL.
		
		\begin{lemma}\label{lemma:long}
			For every $\varepsilon > 0$, there exist positive constants $b_\varepsilon$ and $d_\varepsilon$ such that the following holds. Let $G$ be a graph with maximum degree $\Delta$ and let $\psi : E \to \mathscr{C}$ be a proper edge coloring of $G$. Then there is a proper edge coloring $\varphi:E \to \mathscr{C}\cup\mathscr{C}'$ such that
			\begin{itemize}
				\item $|\mathscr{C}'| = \varepsilon \Delta + o(\Delta)$;
				\item if a cycle is $\varphi$-bichromatic, then it was $\psi$-bichromatic;
				\item there are no $\varphi$-bichromatic cycles of length at least $L$, where $L \coloneqq b_\varepsilon \log\Delta + d_\varepsilon$.
			\end{itemize}
		\end{lemma}
		\begin{proof}
			Let $\mathscr{C}'$ be a set of colors disjoint from $\mathscr{C}$ with $|\mathscr{C}'| = c\Delta$. Fix some $0<p<1$ and construct a random edge coloring $\varphi$ in the following way: For each edge $e \in E(G)$ either do not change its color with probability $1-p$, or choose for it one of the new colors, each with probability $p/|\mathscr{C}'| = p/(c\Delta)$.
			
			Now consider the hypercube digraph $Q \coloneqq Q(E(G))$. Define a set $A \subseteq V(Q)$ by
			$$
				S \in A \,\vcentcolon\Longleftrightarrow\, \text{$\varphi|_S$ satisfies the conditions of the lemma},
			$$
			where $\varphi|_S$ denotes the restriction of $\varphi$ to $G[S]$. The set $A$ is out-closed, $\emptyset \in A$ with probability $1$, and $E(G) \in A$ if and only if $\varphi$ satisfies the conditions of the lemma.
			
			Let $\widetilde{Q}$ be the digraph such that $\widetilde{Q}^s = Q$ and for each $(S \cup \{e\}, S) \in E(Q)$, the edges of $\widetilde{Q}$ between $S \cup \{e\}$ and $S$ are of the following five types:
			\begin{enumerate}
				\item $f_{e'}$ for each edge $e'$ adjacent to $e$;
				\item $f^{\text{II}}_C$ for each $\psi$-bichromatic cycle $C \ni e$ of length at least $L$;
				\item $f^{\text{III}}_C$ for each cycle $C \ni e$ of even length;
				\item $f^{\text{IV}}_C$ for each cycle $C \ni e$ of even length such that $C = e_1$, $e_2$, \ldots, $e_{2t}$ with $e_1 = e$ and $\psi(e_1) = \psi(e_3) = \ldots = \psi(e_{2t-1})$;
				\item $f^{\text{V}}_C$ for each cycle $C \ni e$ of even length such that $C = e_1$, $e_2$, \ldots, $e_{2t}$ with $e_1 = e$ and $\psi(e_2) = \psi(e_4) = \ldots = \psi(e_{2t})$.
			\end{enumerate}
			
			Now define
			$$
				f_{e'} \in F \,\vcentcolon\Longleftrightarrow\, \text{$\varphi(e) = \varphi(e') \in \mathscr{C}'$};
			$$
			$$
				f^\text{II}_C \in F \,\vcentcolon\Longleftrightarrow\,\text{$\varphi(e') = \psi(e')$ for all $e' \in C$};
			$$
			$$
				f^\text{III}_C \in F \,\vcentcolon\Longleftrightarrow\, \text{$C$ is $\varphi$-bichromatic and $\varphi(e') \in \mathscr{C}'$ for all $e' \in C$};
			$$
			$$
				f^\text{IV}_C \in F \,\vcentcolon\Longleftrightarrow\, \text{$C$ is $\varphi$-bichromatic and $\varphi(e_{2k-1}) = \psi(e_{2k-1})$, $\varphi(e_{2k}) \in \mathscr{C}'$ for all $1 \leq k \leq t$};
			$$
			$$
				f^\text{V}_C \in F \,\vcentcolon\Longleftrightarrow\, \text{$C$ is $\varphi$-bichromatic and $\varphi(e_{2k-1}) \in \mathscr{C}'$, $\varphi(e_{2k}) = \psi(e_{2k})$ for all $1 \leq k \leq t$}.
			$$
			It is easy to see that the situations described above exhaust all possible circumstances under which $S \cup \{e\} \not \in A$, even though $S \in A$. Therefore, $F$ is an $A$-cut in $\widetilde{Q}$.
			
			Let us proceed to estimate the risks to the edges of different types. From now on, we assume that $\omega \in \mathbb{R}_{\geq 1}$ is a constant.
			
			{\sc Type 1}. We have
			\begin{align*}
				\Pr(f_{e'} \in F\vert S \setminus \{e'\} \in A) \,=\,\Pr(\varphi(e) = \varphi(e') \in \mathscr{C}' \vert S \setminus \{e'\} \in A) \,\leq\, \frac{p^2}{c\Delta}.
			\end{align*}
			Also,
			$$
				\underline{\omega}(S \cup \{e\}, S \setminus \{e'\}) \leq \omega^2,
			$$
			so
			$$
				\rho_\omega(f_{e'}) \leq \rho_\omega(f_{e'}, S \setminus \{e'\}) \leq \frac{p^2\omega^2}{c\Delta}.
			$$
			Since there are less than $2\Delta$ edges adjacent to any given edge $e$, the edges of the first type contribute at most
			\begin{align}\label{eq:case1}
				2\Delta \cdot \frac{p^2\omega^2}{c\Delta}= 2c \left(\frac{p\omega}{c}\right)^2
			\end{align}
			to the right-hand side of (\ref{eq:LCL}).
			
			{\sc Type 2}. We have
			\begin{align*}
				\Pr\left(f^\text{II}_C \in F \middle\vert S\setminus C \in A\right) \,=\, \Pr\left(\text{$\varphi(e') = \psi(e')$ for all $e' \in C$}\,\middle\vert S \setminus C \in A\right)\,\leq\, (1-p)^{|C|}.
			\end{align*}
			Also,
			$$
				\underline{\omega}(S \cup \{e\}, S \setminus C) \leq \omega^{|C|},
			$$
			so
			$$
				\rho_\omega(f^\text{II}_C) \leq \rho_\omega(f^\text{II}_C, S\setminus C) \leq (1-p)^{|C|}\omega^{|C|}.
			$$
			If we further assume that $(1-p)\omega < 1$, then
			$$
				\rho_\omega(f^\text{II}_C) \leq ((1-p)\omega)^L.
			$$
			Finally, note that there are less than $\Delta$ cycles that contain a given edge $e$ and are $\psi$-bichromatic (because the second edge on such a cycle determines it uniquely). Therefore, the edges of this type contribute at most
			\begin{align}\label{eq:case2}
			\Delta ((1-p)\omega)^L
			\end{align}
			to the right-hand side of (\ref{eq:LCL}).
			
			{\sc Type 3}. We have
			\begin{align*}
				\Pr\left(f^\text{III}_C \in F\middle\vert S \setminus C \in A\right) &= \Pr\left(\text{$C$ is $\varphi$-bichromatic and $\varphi(e') \in \mathscr{C}'$ for all $e' \in C$}\,\middle\vert S \setminus C \in A\right) \\
				&\leq \frac{p^{|C|}}{(c\Delta)^{|C|-2}}.
			\end{align*}
			Also,
			$$
				\underline{\omega}(S \cup \{e\}, S \setminus C) \leq \omega^{|C|},
			$$
			so
			$$
				\rho_\omega(f^\text{III}_C) \leq \rho_\omega(f^\text{III}_C, S \setminus C) \leq \frac{p^{|C|}\omega^{|C|}}{(c\Delta)^{|C|-2}}.
			$$
			There can be at most $\Delta^{2t-2}$ cycles of length $2t$ containing a given edge $e$. Hence, if we assume that $p\omega/c < 1$, then the edges of the third type contribute at most
			\begin{align}\label{eq:case3}
				\sum_{t=2}^\infty \Delta^{2t-2} \cdot \frac{p^{2t}\omega^{2t}}{(c\Delta)^{2t-2}} = c^2 \sum_{t=2}^\infty \left(\frac{p\omega}{c}\right)^{2t} = c^2 \frac{(p\omega/c)^4}{1-(p\omega/c)^2}
			\end{align}
			to the right-hand side of (\ref{eq:LCL}).
			
			{\sc Type 4}. We have
			\begin{align*}
				\Pr\left(f^\text{IV}_C \in F\middle\vert S \setminus C \in A\right) \leq \frac{p^{|C|/2}(1-p)^{|C|/2}}{(c\Delta)^{|C|/2-1}}.
			\end{align*}
			Since
			$$
				\underline{\omega}(S \cup \{e\}, S \setminus C) \leq \omega^{|C|},
			$$
			we get
			$$
				\rho_\omega(f^\text{IV}_C) \leq \rho_\omega(f^\text{IV}_C, S \setminus C) \leq \frac{p^{|C|/2}(1-p)^{|C|/2}}{(c\Delta)^{|C|/2-1}} \cdot \omega^{|C|}.
			$$
			If we further assume that $(1-p)\omega < 1$, then
			$$
				\rho_\omega(f^\text{IV}_C) \leq \frac{p^{|C|/2}\omega^{|C|/2}}{(c\Delta)^{|C|/2-1}}.
			$$
			There can be at most $\Delta^{t-1}$ cycles $C$ of length $2t$ containing a given edge $e$ such that every second edge in $C$ is colored the same by $\psi$. (See Fig. \ref{fig:case4}. Solid edges retain their color from $\psi$ (this color must be the same for all of them). Arrows indicate $t-1$ edges that must be specified in order to fully determine the cycle). Hence, if we assume that $p\omega/c < 1$, then the edges of the fourth type contribute at most
			\begin{align}\label{eq:case4}
				\sum_{t=2}^\infty \Delta^{t-1} \cdot \frac{p^{t}\omega^t}{(c\Delta)^{t-1}} \leq c \sum_{t=2}^\infty \left(\frac{p\omega}{c}\right)^{t} = c \frac{(p\omega/c)^2}{1-p\omega/c}
			\end{align}
			to the right-hand side of (\ref{eq:LCL}).
			
			{\sc Type 5}. The same analysis as for Type 4 (see Fig. \ref{fig:case5}) shows that the contribution of the edges of this type to the right-hand side of (\ref{eq:LCL}) is at most
			\begin{align}\label{eq:case5}
				c \frac{(p\omega/c)^2}{1-p\omega/c},
			\end{align}
			provided that $(1-p)\omega < 1$ and $p\omega/c < 1$.
			
			\begin{figure}[h]
				\centering
				\begin{minipage}{.5\textwidth}
					\centering
					\includegraphics[width=.5\linewidth]{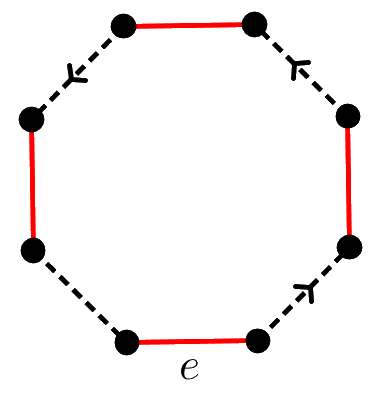}
					\caption{Type 4}
					\label{fig:case4}
				\end{minipage}%
				\begin{minipage}{.5\textwidth}
					\centering
					\includegraphics[width=.5\linewidth]{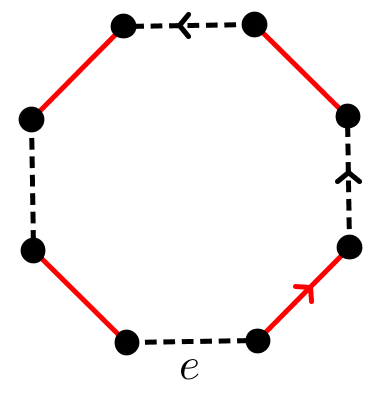}
					\caption{Type 5}
					\label{fig:case5}
				\end{minipage}
			\end{figure}
			
			Adding together (\ref{eq:case1}), (\ref{eq:case2}), (\ref{eq:case3}), (\ref{eq:case4}), and (\ref{eq:case5}), it is enough to have the following inequality:
			\begin{align}\label{eq:all}
				\omega \geq 1 + 2c \left(p\omega/c\right)^2 + \Delta ((1-p)\omega)^L + c^2 \frac{(p\omega/c)^4}{1-(p\omega/c)^2} + 2c \frac{(p\omega/c)^2}{1-p\omega/c},
			\end{align}
			under the assumptions that $(1-p)\omega < 1$ and $p\omega/c < 1$. Denote $y \coloneqq p\omega/c$. Then (\ref{eq:all}) turns into
			\begin{align*}
				\frac{c}{p} \geq \frac{1}{y} + 2c y + c^2 \frac{y^3}{1-y^2} + 2c \frac{y}{1-y} + \frac{\Delta}{y}\left(\frac{c(1-p)}{p}y\right)^L,
			\end{align*}
			and we have the conditions $y < 1$ and $y < p/(c(1-p))$. Let $c = \varepsilon$. We can assume that $\varepsilon$ satisfies
			$$
				2\varepsilon^2 + \frac{\varepsilon^5}{1-\varepsilon^2} + \frac{2\varepsilon^2}{1-\varepsilon} \leq \frac{\varepsilon}{4}.
			$$		
			Take $y = \varepsilon$. Then it is enough to have
			\begin{equation}\label{eq:subst}
				\frac{\varepsilon}{p} \geq \frac{1}{\varepsilon} + \frac{\varepsilon}{4} + \frac{\Delta}{\varepsilon}\left(\frac{\varepsilon^2(1-p)}{p}\right)^L.
			\end{equation}
			Let $p_\varepsilon \coloneqq \varepsilon/(\varepsilon/2+1/\varepsilon)$. Note that
			$$
				\frac{p_\varepsilon}{\varepsilon (1 - p_\varepsilon)} = \frac{1}{\left(\frac{\varepsilon}{2} + \frac{1}{\varepsilon}\right)\left(1-\left.\varepsilon\middle/\left(\frac{\varepsilon}{2} + \frac{1}{\varepsilon}\right)\right.\right)} = \frac{1}{\frac{1}{\varepsilon} - \frac{\varepsilon}{2}} > \varepsilon,
			$$
			so this choice of $p_\varepsilon$ does not contradict our assumptions. 
			Then (\ref{eq:subst}) becomes
			$$
				\frac{\varepsilon}{4} \geq \frac{\Delta}{\varepsilon}\left(\frac{\varepsilon^2(1-p_\varepsilon)}{p_\varepsilon}\right)^L,
			$$
			which is true provided that
			$$
				L \geq \left(\log\left(\frac{p_\varepsilon}{\varepsilon^2 (1-p_\varepsilon)}\right)\right)^{-1} \left(\log \Delta + \log \frac{4}{\varepsilon^2}\right),
			$$
			and we are done.	
		\end{proof}
		
		\subsection{Finishing the proof}
		
		To finish the proof of Theorem \ref{theo:bigg}, fix $\varepsilon > 0$. By Lemma \ref{lemma:short}, if $g(G) > 2r$, where $r \geq 2$, then there is a proper edge coloring $\psi$ of $G$ using at most $(2+\varepsilon/2)\Delta + o(\Delta)$ colors that contains no bichromatic cycles of length at most $L_1 \coloneqq 2a_{\varepsilon/2} (r-2)\log\Delta + 2$. Applying Lemma \ref{lemma:long} to this coloring gives a new coloring $\varphi$ that uses at most $(2+\varepsilon)\Delta + o(\Delta)$ colors and contains no bichromatic cycles of length at most $L_1$ (because there were no such cycles in $\psi$) and at least $L_2 \coloneqq b_{\varepsilon/2} \log\Delta + d_{\varepsilon/2}$. If $r -2 > b_{\varepsilon/2}/(2a_{\varepsilon/2})$ and $\Delta$ is large enough, then $L_1 > L_2$, and $\varphi$ must be acyclic. This observation completes the proof.
		
		\section{Concluding remarks}
		
		We conclude with some remarks on why it seems difficult to get closer to the desired bound $a'(G)\leq \Delta(G)+2$ using the same approach as in the proof of Theorem~\ref{theo:bigg}. Observe that in the proof of Theorem~\ref{theo:bigg} (specifically in the proof of Lemma~\ref{lemma:short}) we reserve $2\Delta$ colors for making a coloring proper and use only $c\Delta$ ``free'' colors to make this coloring acyclic. Essentially, Theorem~\ref{theo:bigg} asserts that $c$ can be made as small as $\varepsilon+o(1)$, provided that $g(G)$ is large enough. It means that the only way to improve the linear term in our bound is to reduce the number of reserved colors, in other words, to implement in the proof some Vizing-like argument. Unfortunately, we do not know how to prove Vizing's theorem by a relatively straightforward application of the LLL (or any analog of it). On the other hand, as was mentioned in the introduction, using a more sophisticated technique (similar to the one used by Kahn \cite{Kahn} in his celebrated proof that every graph is $(1+o(1))\Delta$-edge-choosable), Cai et al.~\cite{Cai} managed to obtain the bound $a'(G) \leq (1+\varepsilon)\Delta + o(\Delta)$, which is very close to the desired $a'(G)\leq \Delta(G)+2$.
		
		\paragraph{Acknowledgments.}
		
		I would like to thank Louis Esperet and R\'{e}mi de Verclos for the observation that the proof of Theorem~\ref{theo:no4} works not only for excluding $C_4$, but for any bipatite graph $H$ as well. I am also grateful to the anonymous referee for his or her valuable comments.
		
		This work is supported by the Illinois Distinguished Fellowship.

	\end{document}